\RequirePackage{ifpdf}
\ifpdf 
\documentclass[pdftex]{article_p}
\else
\documentclass{article_p}
\fi

\usepackage[pagewise]{lineno}

\usepackage{amsmath}
\usepackage{amsfonts}
\usepackage{amssymb}
\usepackage{graphicx}%
\usepackage[mathscr]{eucal} 
\usepackage[new]{old-arrows}
\usepackage{hyperref}
\usepackage{lmodern}
\usepackage[framemethod=tikz]{mdframed}
\usepackage{academicons}
\usepackage{xcolor}

\makeatletter
\newcommand{\Xbar}{{\mathchoice
     {\smash@bar\textfont\displaystyle{0.55}{2.5}\mathscr{X}}
     {\smash@bar\textfont\textstyle{0.55}{2.5}\mathscr{X}}
     {\smash@bar\scriptfont\scriptstyle{0.55}{2.5}\mathscr{X}}
     {\smash@bar\scriptscriptfont\scriptscriptstyle{0.55}{2.5}\mathscr{X}}
          }}
\newcommand{\smash@bar}[4]{%
     \smash{\rlap{\raisebox{-#3\fontdimen5#10}{$\m@th#2\mkern#4mu\mathchar'26$}}}%
          }
\makeatother
\newcommand{\X}[1]{\mathscr{\Xbar}_{#1}}
\newcommand{\D}{\mathscr{D}}
\newcommand{\R}[1]{\mathbb{R}^{#1}}
\newcommand{\T}{\mathsf{T}}
\newcommand{\dd}{\mathrm{d}}
\newcommand{\ii}{\mathbf{i}}
\newcommand{\LL}{\mathscr{L}}
\newcommand{\OO}{\mathscr{O}}
\newcommand{\dlie}[1]{\mathrm{L}_{#1}}

\newcommand{\Cinf}[1]{\mathbf{\mathit{C}}^{\infty}_{#1}}
\newcommand{\ec}[1]{\mbox{$#1$}}

\newtheorem{criterion}[theorem]{Criterion}

\numberwithin{equation}{section}

\allowdisplaybreaks

\begin{document}

\renewcommand{\PaperNumber}{***}

\thispagestyle{empty}


\ArticleName{Infinitesimal Poisson Algebras and Linearization \\ of Hamiltonian Systems}
\ShortArticleName{Infinitesimal Poisson Algebra}

\Author{D. Garc\'ia-Beltr\'an~${}^{1^\dag}$, J. C. Ru\'iz-Pantale\'on~${}^{2^\ddag}\,$ and Yu. Vorobiev~${}^{3^{\S}}$}

\AuthorNameForHeading{Garc\'ia-Beltr\'an, Ru\'iz-Pantale\'on, Vorobiev}

\Address{${\,}^{\dag}$~\,CONACyT Research-Fellow, Departamento de Matem\'aticas, Universidad de Sonora, M\'exico}
\Address{${\,}^{\ddag}$~Instituto de Matem\'aticas, Universidad Nacional Aut\'onoma de M\'exico, M\'exico}
\Address{${\,}^{\S}$~\,Departamento de Matem\'aticas, Universidad de Sonora, M\'exico}

\EmailD{$^{1}$\href{mailto:email@address}{jcpanta\,@im.unam.mx},
        $^{2}$\href{mailto:email@address}{dennise.garcia\,@unison.mx},
        $^{3}$\href{mailto:email@address}{yurimv\,@guaymas.uson.mx}}


\Abstract{Using the notion of a contravariant derivative, we give some algebraic and geometric characterizations of Poisson algebras associated to the infinitesimal data of Poisson submanifolds. We show that such a class of Poisson algebras provides a suitable framework for the study of the Hamiltonization problem for the linearized dynamics along Poisson submanifolds.}


\Classification{53D17, 37J05, 53C05}

    \section{Introduction}

In this paper, we describe a class of Poisson algebras which appear in the context of infinitesimal geometry of Poisson submanifolds, known also as first class constraints \cite{We83, M.C.Marle2000, Zambon}. One of our motivations is to provide a suitable framework for a non-intrinsic Hamiltonian formulation of linearized Hamiltonian dynamics along Poisson submanifolds of nonzero dimension. This question can be viewed as a part of a general Hamiltonization problem for projectable dynamics on fibered manifolds studied in various situations in \cite{Vor04, Vor05, Vorob05, MaRR-91, DavVor-08}. The main feature of our case is that we have to state the Hamiltonization problem in a class of Poisson algebras which do not define any Poisson structures, in general. This situation is related with the problem of the construction of first order approximations of Poisson structures around Poisson submanifolds \cite{Marcut12, Marcut14} which is only well-studied in the case of symplectic leaves \cite{Vor2001, Vor04}.

Let $S$ be an embedded Poisson submanifold of a Poisson manifold \ec{(M,\{,\}_{M})}. Then, for every \,\ec{H \in \Cinf{M}},\, the Hamiltonian vector field \ec{X_{H}} on $M$ is tangent to $S$ and hence can be linearized along $S$. The linearized procedure for \ec{X_{H}} at $S$ leads to a linear vector field \,\ec{\mathrm{var}_{S}X_{H} \in \X{\mathrm{lin}}(E)}\, on the normal bundle of $S$ defined as a quotient vector bundle \,\ec{E=\T_{S}M/\T{S}}.\, In the zero-dimensional case, when \,\ec{S=\{q\}}\, is a \emph{singular point} of the Poisson structure on $M$, the linear vector field \ec{\mathrm{var}_{S}X_{H}} is Hamiltonian relative to the induced Lie-Poisson bracket on \,\ec{E=\T_{q}M}.\, If \,\ec{\dim S > 0},\, then the linearized dynamical model associated to \ec{\mathrm{var}_{S}X_{H}}, called a \textit{first variation system}, does not inherit any natural Hamiltonian structure from the original Hamiltonian system.

This fact gives rise to the so-called Hamiltonization problem for \ec{\mathrm{var}_{S}X_{H}} which is formulated in a class of Poisson algebras on the space of fiberwise affine functions \ec{\Cinf{\mathrm{aff}}(E)} on $E$. In general, this setting can not be extended to the level of Poisson structures on $E$, because of the following observation due to I. M\u{a}rcut \cite{Marcut12}: a first-order local model for the Poisson structure around the Poisson submanifold $S$ does not always exists. For example, a linearized Poisson model exists in the special case when $S$ is a symplectic leaf \cite{Vor2001}.

By using the infinitesimal data of the Poisson submanifold $S$, we introduce a family of Poisson algebras on \ec{\Cinf{\mathrm{aff}}(E)} whose Lie brackets \ec{\{,\}^{\LL}} are parameterized by transversals $\LL$ of $S$, that is, by subbundles of \ec{\T_{S}M} complementary to $\T{S}$. These algebras are called \emph{infinitesimal Poisson algebras} and, in fact, are independent of $\LL$ modulo isomorphisms. For every $\LL$, the first variation system defines a derivation of the corresponding Poisson algebra. We derive the following criterion for the existence of a Hamiltonian structure for the first variation system of \ec{X_{H}} relative to the underlying class of Poisson algebras.

\begin{criterion}\label{criterion}
If the flow of the Hamiltonian vector field \ec{X_{H}} admits an invariant transversal \,\ec{\LL \subset \T_{S}M}\, of the Poisson submanifold $S$,
    \begin{equation}\label{I1}
        \big(\dd_{q}\mathrm{Fl}_{X_{H}}^{t}\big)(\LL_{q}) \,=\, \LL_{\mathrm{Fl}_{X_{H}}^{t}(q)}, \quad \forall\,q \in S,
    \end{equation}
then the first variation system \ec{\mathrm{var}_{S}X_{H}} is a Hamiltonian derivation of the corresponding infinitesimal Poisson algebra,
    \begin{equation*}
        \dlie{\mathrm{var}_{S}X_{H}}(\cdot) \,=\, \{\phi_{H},\cdot\}^{\LL},
    \end{equation*}
for a certain \,\ec{\phi_{H} \in \Cinf{\mathrm{aff}}(E)}.\, The converse is also true.
\end{criterion}

In the case, when $S$ is a symplectic leaf, this criterion is valid in a class of Poisson structures around $S$, called coupling Poisson structures \cite{Vor04, Vorob05}. Here, we also give an application of this result to the linearization of Hamiltonian group actions at $S$. An interesting question is to extend such a criterion to general Poisson submanifolds using, for example, an approach developed in \cite{FloresPantaYura}, results of \cite{Marcut12, Marcut14} and the recent unpublished results on the existence of local models by R. Fernandes and I. Marcut (available at \href{http://www.unige.ch/math/folks/nikolaev/assets/files/GP-20200409-RuiFernandes.pdf}{http://www.unige.ch/math/folks/nikolaev/assets/files/GP-20200409-RuiFernandes.pdf}).

The paper is organized as follows. In Section 2, we recall the definitions of Poisson submanifolds and their infinitesimal data. In Section 3, we describe a class of infinitesimal Poisson algebras on the space of fiberwise affine functions \ec{C_{\operatorname{aff}}^{\infty}(E)} and formulate a result on the first order approximation of the original Poisson algebra around a Poisson submanifold. In Section 4, we show that a factorization of the Jacobi identity for the infinitesimal Poisson algebras leads to their parametrization by means of the so-called Poisson triples involving contravariant derivatives. In Section 5, we give a proof of the first order approximation result which is based on a correspondence between the Poisson triples and the transversal subbundles over a Poisson submanifold. In Section 6, we recall a linearization procedure for dynamical systems at an invariant submanifold which gives a class of projectable vector fields on the normal bundle determining the first variation systems. Section 7 is devoted to the Hamiltonization problem for first variation systems over a Poisson submanifold. First, we derive a geometric criterion for the existence of Hamiltonian structures and then, give its analytic version formulated as the solvability condition of an associated linear nonhomogeneous differential equation. Finally, in Section 8, we apply the Hamiltonization criterion to the construction of linearized models for Hamiltonian group actions around symplectic leaves.

    \section{Preliminaries}\label{preliminaries}

Here, we recall some facts about Poisson submanifolds; for more details see \cite{We83, M.C.Marle2000, Zambon}.

Let \ec{(M,\Pi)} be a Poisson manifold equipped with a Poisson bivector field \,\ec{\Pi \in \Gamma\wedge^{2}\T{M}}\, and the Poisson bracket
    \begin{equation*}
        \{f,g\}_{M} \,=\, \Pi(\dd{f},\dd{g}), \quad f,g \in \Cinf{M}.
    \end{equation*}
An (immersed) submanifold \,\ec{\iota: S \hookrightarrow M}\, is said to be a \emph{Poisson submanifold} of $M$ if the \emph{Poisson bivector field $\Pi$ is tangent} to $S$:
    \begin{equation}\label{Po1}
        \Pi_{q} \in \wedge^{2}\T_{q}S, \quad \forall\, q \in S.
    \end{equation}
This means that $S$ inherits a (unique) Poisson structure \,\ec{\Pi_{S} \in \Gamma\wedge^{2}\T{S}}\, such that the inclusion $\iota$ is a Poisson map. The corresponding Poisson bracket is denoted by
    \begin{equation*}
        \big\{\bar{f},\bar{g}\big\}_{S} \,:=\, \Pi_{S}\big(\dd{\bar{f}},\dd{\bar{g}}\big), \quad \bar{f},\bar{g} \in \Cinf{S}.
    \end{equation*}

There are several equivalent characterizations of when a submanifold is Poisson. Consider the induced bundle morphism \,\ec{\Pi^{\natural}:\T^{\ast}M \rightarrow \T{M}}\, defined by \,\ec{\alpha \mapsto \Pi^{\natural}(\alpha) := \ii_{\alpha}\Pi},\, and denote by \ec{{\T{S}}^{\circ}} the annihilator of $\T{S}$. Then, condition (\ref{Po1}) can be
reformulated in one of the following ways:
    \begin{equation}\label{Po2}
        \Pi^{\natural}\big({\T{S}}^{\circ}\big) = \{0\} \qquad \text{or} \qquad \Pi^{\natural}\big(\T_{S}^{\ast}M\big) \,\subseteq\, \T{S}.
    \end{equation}
This implies that every Hamiltonian vector field \,\ec{X_{H} = \Pi^{\natural}\dd{H}}\, is \emph{tangent} to $S$. Moreover, if $S$ is an embedded submanifold, then the first condition in (\ref{Po2}) is equivalent to the following: the vanishing ideal \ec{I(S) = \left\{f \in \Cinf{M} \,|\, f|_{S} = 0\right\}} is also an \emph{ideal in the Lie algebra} \ec{(\Cinf{M},\{,\}_{M})}.

Symplectic leaves are the simplest type of Poisson submanifolds. If $S$ is a symplectic leaf of $\Pi$ (i.e., a maximal integral manifold of the characteristic foliation), then \,\ec{\Pi^{\natural}(\T_{S}^{\ast}M) = \T{S}}.\, In this case, the Poisson tensor \ec{\Pi_{S}} is \emph{nondegenerate} and defines a symplectic form \ec{\omega_{S}} on $S$,
    \begin{equation}\label{SS}
        \omega_{S}^{\flat} \,=\, -\big(\Pi_{S}^{\natural}\big)^{-1}.
    \end{equation}

In general, a Poisson submanifold $S$ is the union of open subsets of the symplectic leaves of $\Pi$.

Now, consider the \emph{cotangent Lie algebroid} of the Poisson manifold \ec{(M,\Pi)}:
    \begin{equation}\label{cotangentalgebroid}
        A \,:=\, \left(\T^{\ast}M, [,]_{A}, \Pi^{\natural}: \T^{\ast}M \rightarrow \T{M} \right),
    \end{equation}
where
    \begin{equation*}
        [\alpha,\beta]_{A} \,:=\, \ii_{\Pi^{\natural}(\alpha)}\dd\beta \,-\, \ii_{\Pi^{\natural}(\beta)} \dd\alpha \,-\, \dd\langle \alpha,\Pi^{\natural}(\beta) \rangle
    \end{equation*}
is the Lie bracket for 1-forms on $M$.

The key property is that the cotangent Lie algebroid $A$ \eqref{cotangentalgebroid} admits a natural restriction to the Poisson submanifold $S$ in the sense that there exists a Lie algebroid \ec{A_{S}} over $S$,
    \begin{equation*}
        A_{S} \,:=\, \left(\T_{S}^{\ast}M, [,]_{A_{S}}, \Pi^{\natural}|_{S}:\T_{S}^{\ast}M \rightarrow \T{S} \right),
    \end{equation*}
such that the restriction map \,\ec{\Gamma\,\T^{\ast}M \rightarrow \Gamma\,\T_{S}^{\ast}M}\, is a \emph{Lie algebra homomorphism}. Here, the restrictions of the Lie bracket and the anchor are well-defined because of the property that the Poisson tensor $\Pi$ is tangent to $S$.

We observe that there exists a short exact sequence of Lie algebroids
    \begin{equation*}
        0 \longrightarrow {\T{S}}^{\circ} \longrightarrow A_{S} \longrightarrow \T^{\ast}S \longrightarrow 0,
    \end{equation*}
where \ec{\T^{\ast}S} is the cotangent Lie algebroid of \ec{(S,\Pi_{S})} and \ec{{\T{S}}^{\circ}} is a Lie algebroid with zero anchor. The last fact is a consequence of property (\ref{Po2}) which reads as
    \begin{equation*}
        {\T{S}}^{\circ} \,\subseteq\, \ker\big(\Pi^{\natural}|_{S}\big).
    \end{equation*}
It follows also that the annihilator \ec{{\T{S}}^{\circ}} is an \emph{ideal} in \ec{A_{S}}.

So, follow \cite{IKV-98, Marcut12, Marcut14}; by the \emph{infinitesimal data} of the Poisson submanifold $S$ we will mean the restricted Lie algebroid \ec{A_{S}}. In the case when $S$ is a symplectic leaf, \ec{A_{S}} is a transitive Lie algebroid \cite{Mac, Vor04, DVY-12}.

    \section{Infinitesimal Poisson Algebras}\label{infinitesimal Poisson algebras}

Suppose we start with an embedded Poisson submanifold \ec{(S,\Pi_{S})} of a Poisson manifold \ec{(M,\Pi)}. By using the infinitesimal data of $S$, our point is to construct a Poisson algebra \ec{P_{1}} which gives a \emph{first-order approximation} to the original one
    \begin{equation}\label{P1}
        P = \big(\Cinf{M},\cdot,\{,\}_{M}\big)
    \end{equation}
in some natural sense.

Consider the normal bundle of $S$
    \begin{equation*}
        E \,:=\, \T_{S}M \,/\, \T{S}, \qquad \pi:E \longrightarrow S,
    \end{equation*}
and the co-normal (dual) bundle \,\ec{E^{\ast} \rightarrow S}.\, Denote by
    \begin{equation}\label{EcProyCoc}
        \nu: \T_{S}M \longrightarrow E
    \end{equation}
the quotient projection.

Consider a $\Cinf{S}$-module \emph{of fiberwise affine} $\Cinf{}$-\emph{functions} on $E$:
    \begin{equation*}
        \Cinf{\mathrm{aff}}(E) \,:=\, \pi^{\ast}\Cinf{S} \oplus \Cinf{\mathrm{lin}}(E) \,\simeq\, \Cinf{S} \oplus \Gamma{E^{\ast}}.
    \end{equation*}
So, every element \,\ec{\phi \in \Cinf{\mathrm{aff}}(E)}\, is represented as
    \begin{equation*}
        \phi \,=\, \pi^{\ast}f + \ell_{\eta} \,\simeq\, f \oplus \eta,
    \end{equation*}
where \,\ec{f \in \Cinf{S}}\, and \,\ec{\eta \in \Gamma{E^{\ast}}}.\, Here we use the canonical identification \,\ec{\ell:\Gamma{E^{\ast}} \rightarrow \Cinf{\mathrm{lin}}(E)}\, given by \,\ec{\ell_{\eta}(z) = \langle \eta_{\pi(z)},z \rangle},\, for \,\ec{z \in E}.\, First, we remark that \ec{\Cinf{\mathrm{aff}}(E)} is a \emph{commutative algebra} with ``infinitesimal'' multiplication
    \begin{equation}\label{LB2}
        \phi_{1} \cdot \phi_{2} \,=\, \pi^{\ast}(f_{1}f_{2}) \,+\, \ell_{(f_{1}\eta_{2} + f_{2}\eta_{1})}
    \end{equation}
or, equivalently,
    \begin{equation}\label{LB1}
        (f_{1} \oplus \eta_{1}) \cdot (f_{2} \oplus \eta_{2}) \,=\, f_{1}f_{2} \oplus (f_{1}\eta_{2} + f_{2}\eta_{1}).
    \end{equation}
Let \,\ec{\iota_{0}:S \hookrightarrow E}\, be the zero section of the normal bundle. Then, we have the canonical splitting
    \begin{equation}\label{CS1}
        \T_{S}E \,=\, \T{S} \oplus E,
    \end{equation}
and the projection \,\ec{\T_{S}E \rightarrow E}\, along $\T{S}$ whose adjoint gives a vector bundle morphism \,\ec{E^{\ast} \rightarrow \T_{S}^{\ast}E}.\, On the other hand, we have the dual decomposition of (\ref{CS1})
    \begin{equation}\label{CS2}
        \T_{S}^{\ast}E \,=\, E^{\circ} \oplus {\T{S}}^{\circ},
    \end{equation}
and the projection \,\ec{\mathrm{pr}:\T_{S}^{\ast}E \rightarrow {\T{S}}^{\circ}}\, along \ec{E^{\circ}}. Then, decomposition (\ref{CS2}) induces the vector bundle isomorphism \,\ec{\chi: E^{\ast} \rightarrow {\T{S}}^{\circ} \varhookrightarrow \T_{S}^{\ast}E}.\, Now, we define a linearization map
    \begin{equation*}
        \mathrm{Aff}: \Cinf{E} \longrightarrow \Cinf{\mathrm{aff}}(E), \qquad F \,\longmapsto\, \mathrm{Aff}(F) \,=\, \pi^{\ast}f + \ell_{\eta},
    \end{equation*}
with \,\ec{f = \iota_{0}^{\ast}F}\, and \,\ec{\eta = \chi^{-1} \circ \mathrm{pr}(\dd{F}|_{S})}.\, Here \,\ec{\dd{F}|_{S} \in \Gamma\,\T_{S}^{\ast}E}\, is the restricted differential of \,\ec{F \in \Cinf{E}}.\, It is easy to see that $\mathrm{Aff}$ is a homomorphism of commutative algebras.

Now, consider the $\Cinf{S}$-module of \emph{fiberwise linear functions} \ec{\Cinf{\mathrm{lin}}(E)} and the $\Cinf{S}$-module isomorphism
    \begin{equation*}
        \Cinf{\mathrm{lin}}(E) \,\overset{\ell^{-1}}{\longrightarrow}\, \Gamma{E^{\ast}} \,\overset{\chi}{\longrightarrow}\, \Gamma\,{\T{S}}^{\circ}.
    \end{equation*}
Then, the bracket on the Lie algebroid \ec{A_{S}} induces an \emph{intrinsic Lie algebra structure} on \ec{\Cinf{\mathrm{lin}}(E)}:
    \begin{equation*}
        \{\varphi_{1},\varphi_{2}\}^{\operatorname{lin}} \,:=\, \ell\circ\chi^{-1}\left(\big[\chi \circ \ell^{-1}(\varphi_{1}),\chi \circ \ell^{-1}(\varphi_{2})\big]_{A_{S}}\right).
    \end{equation*}
This bracket together with trivial (zero) multiplication on \ec{\Cinf{\mathrm{lin}}(E)} defines a Poisson algebra structure.

It is useful also to given an alternative description of \ec{\Cinf{\mathrm{lin}}(E)}. Indeed, for any \,\ec{\eta_{1},\eta_{2} \in \Gamma{E^{\ast}}}\, define the bracket
    \begin{equation}\label{FLB}
        [\eta_{1},\eta_{2}]_{E^{\ast}} \,=\, \chi^{-1}\big(\,[\chi(\eta_{1}),\chi(\eta_{2})]_{A_{S}}\big),
    \end{equation}
which is $\Cinf{S}$-bilinear. This follows from \eqref{Po2}. Therefore, the co-normal bundle \ec{E^{\ast}} over $S$ inherits from \ec{[,]_{A_{S}}} a fiberwise Lie bracket \,\ec{S \ni q \mapsto [,]_{E_{q}^{\ast}}}\, smoothly varying with \,\ec{q \in S}.\, In other hand, the co-normal bundle \ec{E^{\ast}} is a \emph{bundle of Lie algebras} (not necessarily locally trivial). Moreover, this gives rise to a \emph{Lie-Poisson structure} (a \emph{vertical Lie-Poisson tensor}) on $E$.

\begin{example}
If $S$ is a symplectic leaf, then the bundle of Lie algebras \ec{(E^{\ast},[,]_{E^{\ast}})} is \emph{locally trivial} and the corresponding typical fiber is called the \emph{isotropy algebra} of the leaf.
\end{example}

So, taking into account that we have two intrinsic Poisson algebras $\Cinf{S}$ and \ec{\Cinf{\mathrm{lin}}(E)} associated with the Poisson
submanifold $S$, we arrive at the following definition.

\begin{definition}
By an \emph{infinitesimal Poisson algebra (IPA)} we mean a Poisson algebra
    \begin{equation}\label{AF1}
        \big(\Cinf{\mathrm{aff}}(E )= \pi^{\ast}\Cinf{S} \oplus \Cinf{\mathrm{lin}}(E),\cdot,\{,\}^{\mathrm{aff}}\,\big),
    \end{equation}
which consists of the commutative algebra \ec{(\Cinf{\mathrm{aff}}(E),\cdot)} in (\ref{LB2}) and a Lie bracket \ec{\{,\}^{\mathrm{aff}}} on \ec{\Cinf{\mathrm{aff}}(E)} satisfying the conditions:
    \begin{itemize}
      \item [(a)] the natural projection \,\ec{\Cinf{\mathrm{aff}}(E) \rightarrow \Cinf{S}}\, is a Poisson algebra homomorphism,
      \item [(b)] for any \,\ec{\varphi_{1},\varphi_{2} \in \Cinf{\mathrm{lin}}(E)},\, we have
                    \begin{equation*}
                        \{0 \oplus \varphi_{1},0 \oplus \varphi_{2}\}^{\mathrm{aff}} \,=\, 0 \oplus \{\varphi_{1},\varphi_{1}\}^{\mathrm{lin}}.
                    \end{equation*}
    \end{itemize}
\end{definition}

Observe that for any infinitesimal Poisson algebra, we have an short exact sequence of Poisson algebras
    \begin{equation*}
        0 \longrightarrow \Cinf{\mathrm{lin}}(E) \varlonghookrightarrow \Cinf{\mathrm{aff}}(E) \longrightarrow \Cinf{S} \longrightarrow 0,
    \end{equation*}
where \ec{\Cinf{\mathrm{lin}}(E)} is an ideal.

To end this section we give a positive answer to the question on the existence of a first order approximation of the Poisson algebra (\ref{P1}) around an embedded Poisson submanifold.

By an exponential map we mean a diffeomorphism \,\ec{\mathbf{e}:E \rightarrow M}\, from the total space of the normal bundle onto a neighborhood of $S$ in $M$ which is identical on $S$, \,\ec{\mathbf{e}|_{S} = \mathrm{id}_{S}},\, and such that the composition
    \begin{equation*}
        E_{q} \,\varlonghookrightarrow\, \T_{q}E \,\overset{\dd_{q}\mathbf{e}}{\longrightarrow}\, \T_{q}M \,\overset{\nu_{q}}{\longrightarrow}\, E_{q}
    \end{equation*}
is the \emph{identity map} of the fiber \,\ec{E_{q}=\pi^{-1}(q)}\, over \,\ec{q \in S}.\, An exponential map always exists \cite{LibMar-87}.

\begin{theorem}\label{thmIPA}
For every (embedded) Poisson submanifold \,\ec{S \subset M}\, and an exponential map \,\ec{\mathbf{e}:E \rightarrow M},\, there exists an infinitesimal Poisson algebra \,\ec{P_{1} = (\Cinf{\mathrm{aff}}(E),\cdot,\{,\}^{\mathrm{aff}})},\, which is a first order approximation to \,\ec{P = (\Cinf{M},\cdot,\{,\}_{M})}\, around the zero section \,\ec{S \hookrightarrow E},\, in the sense that
    \begin{equation}\label{AP1}
        \{\phi_{1} \circ \mathbf{e}^{-1},\phi_{2} \circ \mathbf{e}^{-1}\}_{M} \circ \mathbf{e} \,=\, \{\phi_{1},\phi_{2}\}^{\mathrm{aff}} + \OO_{2},
    \end{equation}
for all \,\ec{\phi_{1},\phi_{2} \in \Cinf{\mathrm{aff}}(E)}.
\end{theorem}

Observe that condition (\ref{AP1}) can be reformulated as follows: the mapping
    \begin{equation}\label{AP2}
        \operatorname{Aff} \circ\, \mathbf{e}^{\ast}: \Cinf{M} \rightarrow \Cinf{\mathrm{aff}}(E)
    \end{equation}
is a Poisson algebra homomorphism.

The proof of this fact will be given in the next sections.

    \section{Poisson Triples}\label{Poisson triples}

Here, we describe a structure of infinitesimal Poisson algebras by using the notion of a \emph{contravariant derivative} on a vector bundle over a Poisson manifold introduced in \cite{Va90} (see also \cite{Va94, Fer2000}).

Consider the co-normal bundle \ec{E^{\ast}} over the Poisson submanifold \,\ec{S \subset M}.\, Recall that a contravariant derivative $\D$ on \ec{E^{\ast}} consists of $\R{}$-linear operators \,\ec{\D_{\alpha}:\Gamma{E^{\ast}} \rightarrow \Gamma{E^{\ast}}}\, which are $\Cinf{S}$-linear in \,\ec{\alpha \in \Gamma\,\T^{\ast}S}\, and satisfy the Leibniz-type rule
    \begin{equation*}
        \D_{\alpha}(f\eta) \,=\, f\D_{\alpha}(\eta) \,+\, \big(\dlie{\Pi_{S}^{\natural}(\alpha)}f\big)\eta.
    \end{equation*}
for \,\ec{f \in \Cinf{S}}, \,\ec{\eta \in \Gamma{E^{\ast}}}.\, The curvature \ec{\mathrm{Curv}^{\D}} of $\D$ is defined as
    \begin{equation*}
        \mathrm{Curv}^{\D}(\alpha_{1},\alpha_{2}) \,:=\, [\D_{\alpha_{1}},\D_{\alpha_{2}}] - \D_{[\alpha_{1},\alpha_{2}]_{\T^{\ast}S}}.
    \end{equation*}
Here, \ec{[,]_{\T^{\ast}S}} denotes the Lie bracket for 1-forms on the Poisson manifold \ec{(S,\Pi_{S})}.

\begin{remark}
Every covariant derivative (linear connection) \,\ec{\nabla:\Gamma\,\T{S} \times \Gamma{E^{\ast}} \rightarrow \Gamma{E^{\ast}}}\, induces a contravariant derivative $\D$ which is defined as
    \begin{equation}\label{CC1}
        \D_{\alpha} = \nabla_{\Pi_{S}^{\natural}(\alpha)},
    \end{equation}
and satisfies the following property:
    \begin{equation}\label{CC2}
        \Pi_{S}^{\natural}(\alpha) = 0 \quad \Longrightarrow \quad \D_{\alpha} =\, 0.
    \end{equation}
In general, condition (\ref{CC2}) does not imply the existence of a covariant derivative satisfying (\ref{CC1}) (for more details, see \cite{Fer2000}).
\end{remark}

Now, suppose we are given a triple \ec{\big([,]_{E^{\ast}},\D,\mathscr{K}\big)} consisting of
    \begin{itemize}
    \item the fiberwise Lie algebra bracket \ec{[,]_{E^{\ast}}} on \ec{E^{\ast}} given by (\ref{FLB}),
    \item a contravariant derivative \,\ec{\D:\Gamma\,\T^{\ast}S \times \Gamma{E^{\ast}} \rightarrow \Gamma{E^{\ast}}}\, on the co-normal bundle \ec{E^{\ast}} over the Poisson manifold \ec{(S,\Pi_{S})},
    \item a $\Cinf{S}$-bilinear antisymmetric mapping \,\ec{\mathscr{K}: \Gamma\,\T^{\ast}S \times \Gamma\,\T^{\ast}S \rightarrow \Gamma{E^{\ast}}}.
    \end{itemize}

Assume that the triple \ec{\big([,]_{E^{\ast}},\D,\mathscr{K}\big)} satisfies the following conditions:
    \begin{align}
            & \hspace{0.88cm} [\D_{\alpha},\mathrm{ad}_{\eta}] \,=\, \mathrm{ad}_{\D_{\alpha}\eta}, \label{Ja1} \\[0.15cm]
            & \hspace{0.88cm} \mathrm{Curv}^{\D}(\alpha,\beta) \,=\, \mathrm{ad}_{\mathscr{K}(\alpha,\beta)}, \label{Ja3} \\[0.15cm]
            & \underset{(\alpha,\beta,\gamma)}{\mathfrak{S}} \D_{\alpha}\mathscr{K}(\beta,\gamma) \,+\, \mathscr{K}(\alpha,[\beta,\gamma]_{\T^{\ast}S}) \,=\, 0, \label{Ja2}
    \end{align}
for all \,\ec{\alpha,\beta,\gamma \in \Gamma\,\T^{\ast}S}, \,\ec{\eta \in \Gamma{E^{\ast}}}.\, Here, \,\ec{\mathrm{ad}_{\eta}(\cdot) := [\eta,\cdot]_{E^{\ast}}}.

\begin{definition}
A setup \ec{\big([,]_{E^{\ast}},\D,\mathscr{K}\big)} satisfying (\ref{Ja1})-(\ref{Ja2}) is said to be a \emph{Poisson triple} of a Poisson submanifold \ec{(S,\Pi_{S})} in \ec{(M,\Pi)}.
\end{definition}

Here we arrive at the basic fact.

\begin{lemma}\label{LemaTripleToAlg}
Every Poisson triple \ec{\big([,]_{E^{\ast}},\D,\mathscr{K}\big)} of a Poisson submanifold \ec{S \subset M} induces an infinitesimal
Poisson algebra \,\ec{(\Cinf{\mathrm{aff}}(E) \simeq \Cinf{S} \oplus \Gamma{E^{\ast}},\cdot,\{,\}^{\mathrm{aff}})}\, with multiplication (\ref{LB1}) and the Lie bracket given by
    \begin{equation}\label{PLB1}
        \{f_{1} \oplus \eta_{1},f_{2} \oplus \eta_{2}\}^{\mathrm{aff}} \,:=\, \{f_{1},f_{2}\}_{S} \oplus \big( \D_{\dd{f_{1}}}\eta_{2} - \D_{\dd{f_{2}}}\eta_{1} + [\eta_{1},\eta_{2}]_{E^{\ast}} + \mathscr{K}(\dd{f_{1}},\dd{f_{2}}) \big).
    \end{equation}
\end{lemma}

The proof of this fact is a direct verification that conditions (\ref{Ja1})-(\ref{Ja2}) give a factorization of the Jacobi identity for bracket (\ref{PLB1}).

Using formula (\ref{PLB1}), one can show that the converse is also true; that is, each infinitesimal Poisson algebra induces a Poisson triple.

\begin{corollary}\label{corIPAPT}
There is a one-to-one correspondence between infinitesimal Poisson algebras and Poisson triples.
\end{corollary}

\begin{example}
Consider a Poisson triple \ec{\big([,]_{E^{\ast}},\D,\mathscr{K}\big)} in the case when the fiberwise Lie algebra on \ec{E^{\ast}} is abelian and the contravariant derivative is flat, \,\ec{[,]_{E^{\ast}} \equiv 0}\, and \,\ec{\mathscr{K} = 0}.\, Then, $\D$ is related with the notion of a Poisson module (see \cite{Bursztyn}) and defines the Lie bracket of the form
    \begin{equation*}
        \{f_{1} \oplus \eta_{1},f_{2} \oplus \eta_{2}\}^{\mathrm{aff}} \,=\, \{f_{1},f_{2}\}_{S} \oplus \left(\D_{\dd{f_{1}}} \eta_{2} - \D_{\dd{f_{2}}}\eta_{1}\right).
    \end{equation*}
\end{example}

\begin{remark}
The notion of Poisson triples can be generalize to the more general situation, starting with a module over an abstract Poisson algebra.
One can extend Corollary \ref{corIPAPT} to this case by using the correspondence between Poisson algebras and Lie algebroids \cite{Marcut12, DVY-12, DVYu-15}.
\end{remark}

    \section{Existence of Infinitesimal Poisson Algebra}\label{existence of infinitesimal Poisson algebra}

In this section, we prove the existence of an infinitesimal Poisson algebra structure on the commutative algebra \ec{\Cinf{\mathrm{aff}}(E)} of fiberwise affine functions on the normal bundle $E$ of an embedded Poisson submanifold \ec{(S,\Pi_{S})} in a Poisson manifold \ec{(M,\Pi)}. According to Lemma \ref{LemaTripleToAlg}, it suffices to show that there exists a Poisson triple of $S$.

Pick a splitting
    \begin{equation}\label{ND1}
        \T_{S}M \,=\, \T{S} \oplus \LL,
    \end{equation}
where \,\ec{\LL \subset \T_{S}M}\, is a subbundle complementary to $\T{S}$, called a \emph{transversal} of $S$. Consider also the dual decomposition
    \begin{equation}\label{ND2}
        \T_{S}^{\ast}M \,=\, \LL^{\circ} \oplus {\T{S}}^{\circ},
    \end{equation}
and the quotient projection \,\ec{\nu:\T_{S}M \rightarrow E}\, (\ref{EcProyCoc}). Then, the image of the adjoint morphism \,\ec{\nu^{\ast}:E^{\ast} \rightarrow \T_{S}^{\ast}M}\, is \,\ec{\nu^{\ast}(E^{\ast}) = {\T{S}}^{\circ} \varhookrightarrow \T_{S}^{\ast}M}\, and hence \ec{\nu^{\ast}} gives a vector bundle isomorphism between \ec{E^{\ast}} and \ec{{\T{S}}^{\circ}}. Moreover, decomposition (\ref{ND2}) induces the vector bundle isomorphism \,\ec{\tau_{\LL}: \T^{\ast}S \rightarrow \LL^{\circ}}.

Denote by \,\ec{\varrho_{\LL}:\T_{S}^{\ast}M \rightarrow {\T{S}}^{\circ}}\, the projection along \ec{\LL^{\circ}} according to the decomposition (\ref{ND2}).

\begin{lemma}\label{LemaTransversalTriple}
Every transversal $\LL$ of $S$ induces a Poisson triple
    \begin{equation}\label{TR1}
        \big(\,[,]_{E^{\ast}},\D = \D^{\LL},\mathscr{K} = \mathscr{K}^{\LL}\big),
    \end{equation}
where the contravariant derivative $\D$ and tensor filed $\mathscr{K}$ are given by
    \begin{equation}\label{SU1}
        \nu^{\ast}(\D_{\alpha}\eta) \,:=\, [\tau_{\LL}(\alpha),\nu^{\ast}(\eta)]_{A_{S}},
    \end{equation}
and
    \begin{equation}\label{SU2}
        \nu^{\ast}\big(\mathscr{K}(\alpha,\beta)\big) \,:=\, \varrho_{_{\LL}}\big([\tau_{\LL}(\alpha),\tau_{\LL}(\beta)]_{A_{S}}\big),
    \end{equation}
for all \,\ec{\alpha,\beta \in \Gamma\T^{\ast}S}\, and \,\ec{\eta \in \Gamma{E^{\ast}}}.
\end{lemma}
\begin{proof}
Taking into account that \,\ec{{\T{S}}^{\circ} \subset \T_{S}^{\ast}M}\, is an ideal relative to the Lie bracket \ec{[,]_{A_{S}}}, we get that under the $\LL$-dependent identification
    \begin{equation}\label{ID1}
        \tau_{\LL} \oplus \nu^{\ast}: \T^{\ast}S \oplus E^{\ast} \,\longrightarrow\, \LL^{\circ} \oplus {\T{S}}^{\circ} = \T_{S}^{\ast}M,
    \end{equation}
the triple (\ref{TR1}) transforms to the following one
    \begin{equation}\label{Tr2}
        \big([,]_{{\T{S}}^{\circ}}, \D', \mathscr{K}'\big),
    \end{equation}
where \,\ec{\D':\Gamma\LL^{\circ} \times \Gamma{\T{S}}^{\circ} \rightarrow \Gamma{\T{S}}^{\circ}}\, is a contravariant derivative on the vector bundle \ec{{\T{S}}^{\circ}} given by \,\ec{\D_{\alpha}'\zeta = [\alpha',\zeta]_{A_{S}}},\, for all \,\ec{\alpha' =\tau^{-1}_{\LL}(\alpha) \in \T^{\ast}S, \alpha\in \LL^{\circ}}\, and \,\ec{\zeta \in {\T{S}}^{\circ}}.\, Moreover, the fiberwise Lie bracket \ec{[,]_{{{\T{S}}^{\circ}}}} and the tensor field $\mathscr{K}'$ take the form
    \begin{equation*}
        [\zeta_{1},\zeta_{2}]_{{\T{S}}^{\circ}} \,=\, [\zeta_{1},\zeta_{2}]_{A_{S}}, \qquad \mathscr{K}'(\alpha',\beta') \,=\, \varrho_{_{\LL}}\big([\alpha',\beta']_{A_{S}} \big).
    \end{equation*}
By using identification (\ref{ID1}), one can show that the factorization of the Jacobi identity for the bracket \ec{[,]_{A_{S}}} just leads to the relations like (\ref{Ja1})-(\ref{Ja2}) for triple (\ref{Tr2}). So, this implies that the original triple (\ref{TR1}) is Poisson.
\end{proof}

Combining the above results, we arrive at the following result on the parametrization of infinitesimal Poisson algebras.

\begin{proposition}
Every transversal $\LL$ in (\ref{ND1}) induces an infinitesimal Poisson algebra \,\ec{P_{1}^{\LL} = (\Cinf{\mathrm{aff}}(E),\cdot, \newline \{,\}^{\LL} )},\, where the Lie bracket \ec{\{,\}^{\LL}} is defined by formula (\ref{PLB1}) involving the Poisson triple \ec{([,]_{E^{\ast}},\D^{\LL},\mathscr{K}^{\LL})} (\ref{TR1}). Moreover, the algebra \ec{P_{1}^{\LL}} is independent of $\LL$ up to isomorphism.
\end{proposition}
\begin{proof}
The first assertion follows from Lemma \ref{LemaTripleToAlg} and Lemma \ref{LemaTransversalTriple}. Next, fixing a transversal $\LL$ of $S$, we observe that any another transversal $\widetilde{\LL}$, \,\ec{\T_{S}M = \T{S} \oplus \widetilde{\LL}}\, is represented as follows
    \begin{equation}\label{F1}
        \widetilde{\LL} \,=\, \{w+\delta(w) \,|\, w \in \LL\},
    \end{equation}
where \,\ec{\delta:\LL \rightarrow \T{S}}\, is a vector bundle morphism. On the contrary, for a given $\LL$, an arbitrary vector bundle morphism $\delta$ from $\LL$ to $\T{S}$ induces a transversal $\widetilde{\LL}$ by formula (\ref{F1}). Therefore, we have the following transition rule for the contravariant derivatives \,\ec{\D = \D^{\LL}}\, and \,\ec{\widetilde{\D} = \D^{\widetilde{\LL}}}\, associated with two transversals $\LL$ and $\widetilde{\LL}$ of $S$:
    \begin{equation}\label{RU1}
        \widetilde{\D}_{\alpha} = \D_{\alpha} + \mathrm{ad}_{\mu(\alpha)}.
    \end{equation}
Here \,\ec{\mu:\T^{\ast}S \rightarrow E^{\ast}}\, is a vector bundle morphism of the form
    \begin{equation}\label{RU3}
        \mu \,=\, -\big(\nu|_{\LL}\big)^{\ast - 1} \circ \delta^{\ast}.
    \end{equation}
Moreover, for tensor fields \,\ec{\mathscr{K} = \mathscr{K}^{\LL}}\, and \,\ec{\widetilde{\mathscr{K}} = \mathscr{K}^{\widetilde{\LL}}},\, we also have
    \begin{equation}\label{RU2}
        \widetilde{\mathscr{K}}(\alpha,\beta) \,=\, \mathscr{K}(\alpha,\beta) \,+\, \D_{\alpha}\mu(\beta) \,-\, \D_{\beta}\mu(\alpha) \mu\big([\alpha,\beta]_{\T^{\ast}S}\big) \,+\, [\mu(\alpha),\mu(\beta)]_{E^{\ast}}.
    \end{equation}
Finally, by using transition rules (\ref{RU1}), (\ref{RU2}) and by direct computations, we verify that the transformation \,\ec{f \oplus \eta \mapsto f \oplus (\eta + \mu(\dd{f}))}\, gives an isomorphism between Poisson algebras \ec{P_{1}^{\LL}} and \ec{P_{1}^{\widetilde{\LL}}}.
\end{proof}

To complete the proof of Theorem \ref{thmIPA}, we observe that for a given exponential map \,\ec{\mathbf{e}:E \rightarrow M},\, the algebra \ec{P_{1}^{\LL}} gives a first order approximation to the original one \,\ec{P=\Cinf{M}},\, in the sense of (\ref{AP1}), under the following choice of $\LL$:
    \begin{equation}\label{CC}
        \LL_{q} = \big(\dd_{q}\mathbf{e}\big)\big(E_{q}\big), \quad \forall\,q \in S.
    \end{equation}

\begin{remark}\label{remark:cociente}
As was observed in \cite{Marcut12}, the infinitesimal data of $S$ intrinsically induce the Poisson algebra \,\ec{\Cinf{M}/I^{2}(S)}.\, One can show that \,\ec{P^{\mathscr{L}}_{1}}\, is isomorphic to this Poisson algebra.
\end{remark}

    \section{The Linearization Procedure along Submanifolds}\label{the linearization procedure along submanifolds}

Here, we describe a general linearization procedure for vector fields at invariant submanifolds (see, also \cite{MaRR-91}).

Let $M$ be a $\Cinf{}$ manifold $M$ and \,\ec{S \subset M}\, be an embedded submanifold. Suppose that we are given a vector field $X$ on $M$ which is \emph{tangent} to $S$, \,\ec{X_{q} \in \T_{q}S}, for all \,\ec{q \in M};\, and hence its flow \ec{\mathrm{Fl}_{X}^{t}} leaves $S$ \emph{invariant}. The Lie algebra of such vector fields is denoted by $\X{S}(M)$.

Consider the normal bundle \,\ec{E=\T_{S}M/\T{S}}\, of $S$ with canonical projection \,\ec{\pi:E \rightarrow S}.\, Denote by \ec{\X{\mathrm{lin}}(E)} the Lie algebra of \emph{linear vector fields} on $E$. Each element $V$ of \ec{\X{\mathrm{lin}}(E)} is characterized by the properties: $V$ descends under $\pi$ to a vector field $v$ on $S$, and the Lie derivative $\dlie{V}$ leaves invariant the subspace \ec{\Cinf{\mathrm{lin}}(E)}.

Then, for every linear vector field \,\ec{V \in \X{\mathrm{lin}}(E)},\, the Lie derivative \,\ec{\dlie{V}:\Cinf{\mathrm{aff}}(E) \rightarrow \Cinf{\mathrm{aff}}(E)}\, induces a \emph{derivation} of the commutative algebra \ec{\Cinf{\mathrm{aff}}(E)} with multiplication (\ref{LB2}). It is clear that $\dlie{V}$ leaves invariant the components \ec{\pi^{\ast}\Cinf{S}} and \ec{\Cinf{\mathrm{lin}}(E)} in decomposition (\ref{AF1}).

Denote by \,\ec{\rho_{\varepsilon}:E \rightarrow E}\, the dilation, that is, the fiberwise multiplication on $E$ by a factor \,\ec{\varepsilon > 0}.\, Fix an exponential map \,\ec{\mathbf{e}:E \rightarrow M}\, from the total space onto a
neighborhood of $S$ in $M$. Since \,\ec{\mathbf{e}|_{S} = \mathrm{id}_{S}},\, the pullback vector field \ec{\mathbf{e}^{\ast}X} is tangent to the zero section \,\ec{S \subset E}\, and its restriction to $S$ is just the restriction \,\ec{v := X|_{S}}\, of $X$ to $S$.

Denote \,\ec{\mathbf{e}_{\varepsilon} := \mathbf{e} \circ \rho_{\varepsilon}}.\, Then, one can show that the following limit
    \begin{equation*}\label{DL1}
        \mathrm{var}_{S}X \,:=\, \lim_{\varepsilon \rightarrow 0}\mathbf{e}_{\varepsilon}^{\ast}X \,\in\, \X{\mathrm{lin}}(E)
    \end{equation*}
exists and gives a \emph{linear vector field} on $E$ which descends to the restriction \,\ec{v = X|_{S}}, \,\ec{\dd\pi \circ \mathrm{var}_{S}X = v \circ \pi},\, and is independent of the choice of an exponential map \ec{\mathbf{e}}. It is clear that the zero section \,\ec{S \hookrightarrow E}\, is an invariant submanifold of the vector field \ec{\mathrm{var}_{S}X} whose restriction to $S$ is just $v$.

The linear dynamical system \ec{(E,\mathrm{var}_{S}X,S)} on the normal bundle $E$ is called the \emph{first variation system} of the vector field $X$ over an invariant submanifold \,\ec{S \subset M}.

Observe that the linear vector field \ec{\mathrm{var}_{S}X} gives a $0$th-\emph{order approximation} to $X$ around the submanifold $S$, in the sense that \,\ec{\mathbf{e}_{\varepsilon}^{\ast}X = \mathrm{var}_{S}X + \OO(\varepsilon)}\, as \,\ec{\varepsilon \rightarrow 0}.

Indeed, fix a transversal \,\ec{\LL \subset \T_{S}M}\, of $S$ in (\ref{ND1}) and consider the canonical decomposition (\ref{CS1}). Pick an exponential map \,\ec{\mathbf{e}:E \rightarrow M}\, satisfying the compatibility condition (\ref{CC}). Then, we have the expansion
    \begin{equation}\label{FVS3}
        \mathbf{e}_{\varepsilon}^{\ast}X \,=\, \mathrm{var}_{S}(X) \,+\, \varepsilon\,\mathscr{T} \,+\, \mathscr{O}(\varepsilon^{2}),
    \end{equation}
where the vector field $\mathscr{T}$ on $E$ is uniquely determined by the choice of a transversal $\LL$ in (\ref{ND1}) modulo vertical vector fields on $E$, that is, by elements of \,\ec{\X{V}(E) = \Gamma\,\mathrm{Ver}(E)}.\, Here, \,\ec{\mathrm{Ver}(E) = \ker{\dd\pi}}\, is the vertical subbundle of $E$. The image of the vector field $\mathscr{T}$ in (\ref{FVS3}) under the natural projection \,\ec{\X{E} \rightarrow \X{E}/\X{V}(E)}\, is called the \emph{dynamical torsion} of the vector $X$ relative to a transversal $\LL$ to the invariant submanifold $S$ and denoted by \ec{\mathrm{tor}_{S}(X,\LL)}.

Therefore, the first variation system \ec{(E,\mathrm{var}_{S}X,S)} gives a natural \emph{linearized model} for the original dynamical system \ec{(M,X,S)}.

It is also useful to give a coordinate representation for the linearized model. Let \,\ec{(x,y)=(x^{i},y^{a})}\, be a coordinate system on $E$, where \ec{(x^{i})} are coordinates on $S$ and \ec{(y^{a})} are coordinates along the fibers with respect to a basis \ec{(e_{a})} of local sections of $E$. Then,
    \begin{equation}\label{FVS8}
        v \,=\, v^{i}(x)\frac{\partial}{\partial x^{i}}, \qquad \mathbf{e}^{\ast}X \,=\, X^{i}(x,y)\frac{\partial}{\partial x^{i}} \,+\, X^{a}(x,y)\frac{\partial}{\partial y^{a}},
    \end{equation}
with \,\ec{X^{i}(x,0) = v^{i}(x)}, \,\ec{X^{a}(x,0)=0}.\, So, we have
    \begin{equation*}
        \mathrm{var}_{S}X \,=\, v^{i}(x)\frac{\partial}{\partial x^{i}} \,+\, \frac{\partial X^{a}}{\partial y^{b}}\bigg|_{(x,0)}y^{b}\frac{\partial}{\partial y^{a}},
    \end{equation*}
and
    \begin{equation*}
        \mathscr{T} \,=\, \frac{\partial X^{i}}{\partial y^{a}}\bigg|_{(x,0)}y^{a}\frac{\partial}{\partial x^{i}} \,+\, \frac{1}{2}\frac{\partial^{2}X^{a}}{\partial y^{b}\partial y^{c}}\bigg|_{(x,0)}y^{b}y^{c}\frac{\partial}{\partial y^{a}}.
    \end{equation*}
Therefore, locally, the dynamical torsion is represented as
    \begin{equation}\label{TOR}
        \mathrm{tor}_{S}(X,\LL) \,=\, \frac{\partial X^{i}}{\partial y^{a}}\bigg|_{(x,0)}y^{a}\frac{\partial}{\partial x^{i}}.
    \end{equation}

Recall that a \emph{transversal} $\LL$ of $S$ is said to be $X$-\emph{invariant}, if the subbundle \,\ec{\LL \subset \T_{S}M}\, is invariant under the differential of the flow $X$ (condition (\ref{I1})).

The vanishing of the dynamical torsion has the following meaning.

\begin{lemma}\label{LemmaTransversal}
A transversal $\LL$ of $S$ is $X$-invariant if and only if
    \begin{equation}\label{FVS12}
        \mathrm{tor}_{S}(X,\LL) \,=\, 0.
    \end{equation}
\end{lemma}
\begin{proof}
Fixing an exponential map $\mathbf{e}$ satisfying condition (\ref{CC}), let us consider the pull-back vector field \ec{\mathbf{e}^{\ast}X} on $E$. Then, the $X$-invariance of the transversal $\LL$ is equivalent to the invariance of the splitting \,\ec{\T_{S}E=\T{S} \oplus E}\, with respect to the flow of \ec{\mathbf{e}^{\ast}X}. In infinitesimal terms, the \ec{\mathbf{e}^{\ast}X}-invariance of the subbundle $E$ of \ec{\T_{S}E} is expressed as follows
    \begin{equation}\label{FVS13}
        [\mathbf{e}^{\ast}X,Y]_{q} \,\in\, E_{q} \subset \T_{q}E,
    \end{equation}
for any \,\ec{q \in S}\, and \,\ec{Y \in \X{V}(E)}.\, Taking \,\ec{Y = \frac{\partial}{\partial y^{b}}}\, and by using (\ref{FVS8}), we get
    \begin{equation*}
        \big[\mathbf{e}^{\ast}X,\tfrac{\partial}{\partial y^{b}}\big] \,=\, -\left(\frac{\partial X^{i}}{\partial y^{b}}{(x,y)}\frac{\partial}{\partial x^{i}} + \frac{\partial X^{a}}{\partial y^{b}}{(x,y)}\frac{\partial}{\partial
y^{a}}\right).
    \end{equation*}
It follows that, in local terms, condition (\ref{FVS13}) reads \,\ec{{\partial X^{i}}/{\partial y^{b}}\,|_{(x,0)} = 0},\, for \,\ec{b=1,\ldots, \dim S}.\, Comparing this with (\ref{TOR}), we prove (\ref{FVS12}).
\end{proof}

We conclude this section with the following observation on the symmetry properties of the linearized dynamics over $S$. It follows from
(\ref{FVS3}) that the correspondence
    \begin{equation}\label{FVS15}
        \X{S}(M) \ni X \,\longmapsto\, \mathrm{var}_{S}X \in \X{\mathrm{lin}}(E)
    \end{equation}
is a \emph{Lie algebra homomorphism}, \,\ec{\mathrm{var}_{S}[X_{1},X_{2}] = [\mathrm{var}_{S}X_{1},\mathrm{var}_{S}X_{2}]}.

In context of the symmetries of a given vector field $X$ and its first variation system, we have the following consequence: the image under the homomorphism (\ref{FVS15}) of the Lie algebra of vector fields on $M$ which are tangent to $S$ and commute with $X$ belongs to the Lie algebra of linear vector fields on $E$ commuting with \ec{\mathrm{var}_{S}X}.

Moreover, we have the following fact. For every \,\ec{H \in \Cinf{M}},\, denote by \,\ec{H_{\LL}^{\mathrm{aff}} \in \Cinf{\mathrm{aff}}(E)}\, its \emph{first-order approximation} around $S$, defined by means of homomorphism (\ref{AP2}),
    \begin{equation}\label{LH1}
        H_{\LL}^{\mathrm{aff}} \,:=\, \mathrm{Aff}(H \circ \mathbf{e}) \,=\, \pi^{\ast}h + \ell_{\eta^{\LL}} \,=\, F^{(0)} + F_{\LL}^{(1)}.
    \end{equation}
Here, \,\ec{h = H|_{S}},
    \begin{equation}\label{LH2}
        \eta^{\LL} =\, \chi^{-1} \circ \mathrm{pr}\big(\dd(H \circ \mathbf{e})|_{S}\big),
    \end{equation}
and an exponential map \,\ec{\mathbf{e}:E \rightarrow M}\, is compatible with a given transversal $\LL$ by condition (\ref{CC}).

\begin{lemma}
Let \,\ec{F \in \Cinf{M}}\, be a first integral of a vector field \,\ec{X \in \X{S}(M)}.\, Suppose that a transversal $\LL$ is $X$-invariant. Then, the fiberwise affine function \ec{F_{\LL}^{\mathrm{aff}}} is a first integral of the first variation system \ec{\mathrm{var}_{S}X},
    \begin{equation}\label{FI2}
        \dlie{\mathrm{var}_{S}X}F^{(0)} \,=\, 0 \qquad \text{and} \qquad \dlie{\mathrm{var}_{S}X}F_{\LL}^{(1)} \,=\, 0.
    \end{equation}
\end{lemma}
\begin{proof}
The equality \,\ec{\dlie{X}F=0}\, implies that
    \begin{equation}\label{FVS19}
        \dlie{\mathbf{e}_{\varepsilon}^{\ast}X}\big(\mathbf{e}_{\varepsilon}^{\ast}F\big) \,=\, 0.
    \end{equation}
In particular, \,\ec{F^{(0)}=\pi^{\ast}(\iota_{S}^{\ast}F)}\, is a first integral of the restriction \,\ec{v=X|_{S}}.\, On the other hand, by decomposition (\ref{FVS3}) we get
    \begin{equation}\label{FVS20}
        \dlie{\mathbf{e}_{\varepsilon}^{\ast}X}(\mathbf{e}_{\varepsilon}^{\ast}F) \,=\, \pi^{\ast}\dlie{v}(\iota_{S}^{\ast}F) \,+\, \varepsilon\big(\dlie{\mathrm{var}_{S}X}F_{\LL}^{(1)} + \dlie{\mathscr{T}}F^{(0)}\big) \,+\, \OO(\varepsilon^{2}).
    \end{equation}
The $X$-invariance of the transversal $\LL$ is equivalent to condition (\ref{FVS12}). This means that the vector field $\mathscr{T}$ is vertical and hence \,\ec{\dlie{\mathscr{T}}(\pi^{\ast}f)=0},\, for any \,\ec{f \in \Cinf{S}}.\, Then, (\ref{FI2}) follows from (\ref{FVS19}), (\ref{FVS20}).
\end{proof}

    \section{The Hamiltonization Problem}\label{the hamiltonization problem}

As we mentioned above, the linearization of Hamiltonian dynamics at invariant submanifolds may destroy the Hamiltonian property. This feature of the linearization procedure gives rise to the Hamiltonization problem for linearized models around invariant (Poisson) submanifolds. We study this problem in the class of infinitesimal Poisson algebras described in the previous sections.

Let \ec{(S,\Pi_{S})} be an embedded Poisson submanifold of a Poisson manifold \ec{(M,\Pi)}. Let \,\ec{X_{H}=\ii_{\dd{H}}\Pi}\, be a Hamiltonian vector field on $M$ of a function \,\ec{H \in \Cinf{M}}.\, Then, \ec{X_{H}} is \emph{tangent} to $S$ and its restriction \,\ec{v_{h} = X_{H}|_{S}}\, is a Hamiltonian vector field on \ec{(S,\Pi_{S})}, \,\ec{v_{h}=\ii_{\dd{h}}\Pi_{S}}\, with \,\ec{h=H|_{S}}.

Consider the first variation system \ec{\mathrm{var}_{S}X_{H}} on the normal bundle $E$ of $S$.

To describe the properties of \ec{\mathrm{var}_{S}X_{H}}, let us fix a transversal $\LL$ of $S$ and pick an exponential map
\,\ec{\mathbf{e}:E \rightarrow M}\, satisfying (\ref{FVS20}). Then, by Theorem \ref{thmIPA} and Corollary \ref{corIPAPT}, we have the infinitesimal Poisson algebra \ec{(\Cinf{\mathrm{aff}}(E),\cdot,\{,\}^{\LL})} associated with a Poisson triple \ec{\big([,]_{E^{\ast}},\D^{\LL},\mathscr{K}^{\LL}\big)}.

\begin{lemma}
The first variation system of \ec{X_{H}} over $S$ is a derivation of the infinitesimal Poisson algebra \ec{(\Cinf{\mathrm{aff}}(E),\cdot,\{,\}^{\LL})}, \,\ec{\mathrm{var}_{S}X_{H} \in \mathrm{Der}(\Cinf{\mathrm{aff}}(E))}.
\end{lemma}

The next question is to find out under which conditions for the transversal $\LL$, the derivation \ec{\mathrm{var}_{S}X_{H}} is
Hamiltonian relative to \ec{\{,\}^{\LL}}. We formulate the following criterion for the existence of a Hamiltonian structure for the first variation system.

\begin{theorem}\label{TeoHamPoissAlg}
The first variation system \ec{\mathrm{var}_{S}X_{H}} is a Hamiltonian derivation of the infinitesimal Poisson algebra \ec{(\Cinf{\mathrm{aff}}(E),\cdot,\{,\}^{\LL})} if and only if the transversal $\LL$ to the Poisson submanifold $S$ is \ec{X_{H}}-invariant. In this case, \ec{\mathrm{var}_{S}X_{H}} is Hamiltonian relative to the coupling Lie bracket \ec{\{,\}^{\LL}} (\ref{PLB1}) on \ec{\Cinf{\mathrm{aff}}(E)} associated to the Poisson triple \ec{\big([,]_{E^{\ast}},\D^{\LL},\mathscr{K}^{\LL}\big)} and the fiberwise affine function \ec{H_{\LL}^{\mathrm{aff}}} in (\ref{LH1}),
    \begin{equation}\label{HC1}
        \dlie{\mathrm{var}_{S}X_{H}}\phi \,=\, \{H_{\LL}^{\mathrm{aff}},\phi\}^{\LL}, \quad \forall\,\phi \in \Cinf{\mathrm{aff}}(E).
    \end{equation}
Moreover, if \,\ec{F \in \Cinf{M}},\, is a first integral of the Hamiltonian system \ec{X_{H}}, then its first order approximation \ec{F_{\LL}^{\mathrm{aff}}} is a Poisson commuting first integral of \ec{\mathrm{var}_{S}X_{H}}, \,\ec{\big\{H_{\LL}^{\mathrm{aff}},F_{\LL}^{\mathrm{aff}}\big\}^{\LL} = 0}.
\end{theorem}

As consequence of this theorem, we derive Criterion 1.1.

\begin{corollary}
The existence of a Hamiltonian structure for the first variation system \ec{\mathrm{var}_{S}X_{H}} is provided by the existence of an
invariant splitting (\ref{ND1}) for the original Hamiltonian system.
\end{corollary}

We prove Theorem \ref{TeoHamPoissAlg} in few steps.

Given an arbitrary transversal $\LL$ and an exponential map $\mathbf{e}$ satisfying condition (\ref{CC}), consider the contravariant derivative \,\ec{\D=\D^{\LL}}\, and define the horizontal lift \ec{\mathrm{hor}_{\alpha}^{\D}} of a 1-form \,\ec{\alpha \in \Gamma\,\T^{\ast}S}\, as a linear vector field on $E$ given by \,\ec{\dlie{\mathrm{hor}_{\alpha}^{\D}}\ell_{\eta} = \ell_{\D_{\alpha}\eta}},\, for \,\ec{\eta \in \Gamma{E^{\ast}}}.\, In particular, for \,\ec{\alpha=\dd{f}},\, the horizontal lift \ec{\mathrm{hor}_{\dd{f}}^{\D}} descends to the Hamiltonian vector field \ec{v_{f}} on $S$. Moreover, consider a vertical bivector field \,\ec{\Lambda \in \Gamma\wedge^{2}\mathrm{Ver}(E)}\, which is fiberwise Lie-Poisson structure associated to the Lie bracket \ec{[\eta_{1},\eta_{2}]_{E^{\ast}}}, \,\ec{\Lambda(\dd\ell_{\eta_{1}},\dd\ell_{\eta_{2}}) = \ell_{[\eta_{1},\eta_{2}]_{E^{\ast}}}},\, for any \,\ec{\eta_{1},\eta_{2} \in \Gamma{E^{\ast}}}.

\begin{lemma}
The first variation system admits the following $\LL$-dependent decomposition into horizontal and vertical components
    \begin{equation}\label{HF1}
        \mathrm{var}_{S}X_{H} \,=\, \mathrm{hor}_{\dd{h}}^{\D^{\LL}} +\, \ii_{\dd\ell_{\eta^{\LL}}}\Lambda,
    \end{equation}
where \,\ec{h = H|_{S}}\, and \,\ec{\eta^{\LL} \in \Gamma{E^{\ast}}}\, is defined by (\ref{LH2}).
\end{lemma}

Therefore, formula (\ref{HF1}) shows that under a fixed transversal $\LL$ of $S$, the first variation system \ec{\mathrm{var}_{S}X_{H}} is uniquely determined by the element \,\ec{h \oplus \eta \in \Cinf{S} \oplus \Gamma{E^{\ast}}},\, which is given in local coordinates as
    \begin{equation*}
        \eta^{\mathcal{L}} =\, \eta_{a}^{\LL}e^{a}, \qquad \eta_{a}^{\LL}(x) \,:=\, \frac{\partial(H \circ \mathbf{e})}{\partial y^{a}}{(x,0)},
    \end{equation*}
where \ec{(e^{a})} is the dual basis of local sections of \ec{E^{\ast}}.

\begin{lemma}\label{LemaIntrisicTorsion}
The derivation \ec{\mathrm{var}_{S}X_{H}} is Hamiltonian relative to the Lie bracket \ec{\{,\}^{\LL}} and function \ec{H_{\LL}^{\mathrm{aff}}}, that is, condition (\ref{HC1}) holds, if and only if the element \ec{h \oplus \eta^{\LL}} satisfies the equation
    \begin{equation}\label{HSFVS32}
        \ii_{\dd{h}}\mathscr{K}^{\LL} - \D^{\LL}\eta^{\LL} \,=\, 0.
    \end{equation}
\end{lemma}

This fact follows from the representation (\ref{HF1}) and definition of the Lie bracket \ec{\{,\}^{\LL}}.

Now, let us derive a formula for the torsion term in decomposition (\ref{FVS3}) of \ec{\mathbf{e}_{\varepsilon}^{\ast}X_{H}}. In coordinates \,\ec{(x,y) = (x^{i},y^{a})},\, we have
    \begin{equation*}
        \D_{\dd{x^{j}}}^{\LL}\big(\eta_{a}e^{a}\big) \,=\, \left(\D_{b}^{ja}\eta_{a} + \psi^{ji}\frac{\partial\eta_{b}}{\partial x^{i}}\right)e^{b}, \qquad \mathscr{K}^{\LL}(\dd x^{i},\dd x^{j}) \,=\, \mathscr{K}_{a}^{ij}e^{a},
    \end{equation*}
where \,\ec{\Pi_{S}=\frac{1}{2}\psi^{ij}(x)\frac{\partial}{\partial x^{i}} \wedge \frac{\partial}{\partial x^{j}}}\, is the Poisson tensor on $S$. Moreover, by using these relations and definitions (\ref{SU1}), (\ref{SU2}), for the Poisson tensor \ec{\mathbf{e}_{\varepsilon}^{\ast}\Pi} on $E$, we have the following expansions of the pairwise Poisson brackets:
    \begin{align}
        \{x^{i},x^{j}\}_{E} \,&=\, \psi^{ij}(x) + \varepsilon\,\mathscr{K}_{a}^{ij}(x)y^{a} + \OO(\varepsilon^{2}), \label{EcRelat1}\\
        \{x^{i},y^{a}\}_{E} \,&=\, \varepsilon\,\D_{b}^{ia}(x)y^{b} + \OO(\varepsilon^{2}), \label{EcRelat2} \\
        \{y^{a},y^{b}\}_{E} \,&=\, \tfrac{1}{\varepsilon}\lambda_{c}^{ab}(x)y^{c} + \OO(1). \label{EcRelat3}
    \end{align}
By these relations, we compute the term of order $\varepsilon$ in the expansion of \,\ec{\mathbf{e}_{\varepsilon}^{\ast}X_{H} = (\mathbf{e}_{\varepsilon}^{\ast}\Pi)^{\natural}\dd(H \circ \mathbf{e}_{\varepsilon})}:
    \begin{equation*}
        \mathrm{tor}_{S}(X_{H},\LL) \,=\, \left( \frac{\partial h}{\partial x^{i}}\mathscr{K}_{b}^{ij} - \psi^{ji}\frac{\partial\eta_{b}}{\partial x^{i}} - \D_{b}^{ja}\eta_{a}\right)y^{b}\frac{\partial}{\partial x^{j}}.
    \end{equation*}
It follows from here that condition (\ref{HSFVS32}) means that \,\ec{\mathrm{tor}_{S}(X_{H},\LL)=0}\, and hence by Lemma \ref{LemmaTransversal} it is equivalent to the \ec{X_{H}}-invariance of the transversal $\LL$. Applying Lemma \ref{LemaIntrisicTorsion} ends the proof of Theorem \ref{TeoHamPoissAlg}

\begin{example}\label{exm:contraexam}
Consider the Lie-Poisson bracket on \,\ec{\operatorname{e}^{\ast}(3)=\mathbb{R}^{6}=\mathbb{R}_{w}^{3}\times\mathbb{R}_{z}^{3}}:
    \begin{equation*}
        \{w^{i},w^{j}\} \,=\, \epsilon^{ijk}w_{k}, \qquad \{w^{i},z^{j}\} \,=\, \epsilon^{ijk}z_{k}, \qquad \{z_{i},z_{j}\} \,=\, 0.
    \end{equation*}
The $3$-dimensional submanifold \,\ec{S = \big\{z=0\big\} = \mathbb{R}_{w}^{3}\times\{0\}}\, is a Poisson submanifold where the rank of the Poisson tensor takes values $2$ or $0$. For the transversal $\mathscr{L}$ generated by \ec{{\partial}/{\partial z^{a}}}, \,\ec{a=1,2,3};\, we choose a tubular neighborhood $U$ of $S$ as \,\ec{U = S \times \mathbb{R}_{z}^{3}}\, equipped with coordinates \ec{x=w} and \ec{y=z}. Then, by using relations (\ref{EcRelat1})-(\ref{EcRelat3}), we compute \,\ec{\psi^{ij}(x) = \epsilon^{ijk}x^{k}}\, and the corresponding Poisson triple \,\ec{\mathcal{\D}_{b}^{ia} = \epsilon^{iab}, \mathcal{\mathscr{K}}_{a}^{ij} = 0, \lambda_{c}^{ab} = 0}.\, So, the contravariant derivative $\D$ is flat and the fiberwise Lie algebra is abelian. Moreover, one can show that, in this case, condition (\ref{CC2}) does not hold and hence $\D$ can not be generated by a linear connection in the sense of (\ref{CC1}).
\end{example}

\begin{remark}
Algebraically, Theorem \ref{TeoHamPoissAlg} is based on the following arguments. As we have mentioned in Remark 5.3, for a given transversal $\LL$, the infinitesimal Poisson algebra \,\ec{P^{\LL}_{1} = (\Cinf{S} \oplus \Gamma{E^{\ast}}, \cdot ,\{,\}^{\LL})} is naturally identified with the quotient Poisson algebra \ec{\Cinf{M} / I^{2}(S)}. Every vector field $X$ on $M$ tangent to $S$ induces a derivation \ec{X^{(2)}} of \ec{\Cinf{M} / I^{2}(S)} because it preserves \ec{I^{2}(S)}. In the case when \,\ec{X=X_{H}}, it holds that \ec{X_{H}^{(2)}} is the Hamiltonian derivation of the element \,\ec{H + I^{2}(S) \in \Cinf{M} / I^{2}(S)}.\, Under the above  identification, the derivation \ec{X_{H}^{(2)}} has two components: one that is diagonal acting on \ec{\Cinf{S} \oplus \Gamma{E^{\ast}}}, and one that sends \ec{\Cinf{S}} to \ec{\Gamma{E^{\ast}}} and is induced by the torsion \ec{\mathrm{tor}_{S}(X,\LL)}. Then, \ec{X_{H}^{(2)}} coincides with \ec{\mathrm{var}_{S}X} if and only if the torsion vanishes. Therefore, the torsionless condition implies that, under the identification \,\ec{H + I^{2}(S) = H^{\mathrm{aff}} = h \oplus \eta^{\LL}},\, the derivation \ec{\mathrm{var}_{S}X} is  Hamiltonian relative to \ec{H^{\mathrm{aff}}}.
\end{remark}

It is useful to reformulate the criterion in Theorem \ref{TeoHamPoissAlg}, as the solvability condition of a global differential equation associated with the infinitesimal data of the submanifold $S$.

By (\ref{HSFVS32}) and the transition rules (\ref{F1}), (\ref{RU1}), (\ref{RU2}), we derive the following criterion.

\begin{proposition}
Fix a transversal $\LL$ and consider the element \ec{h \oplus \eta^{\LL}} representing the first variation system
\ec{\mathrm{var}_{S}X_{H}}. If the morphism \,\ec{\mu:\T^{\ast}S \rightarrow E^{\ast}}\, satisfies the equation
    \begin{equation}\label{EcHamVar}
        \big(\ii_{\dd{h}} \circ \D^{\LL} + \mathrm{ad}_{\eta} + \D^{\LL} \circ \ii_{\dd{h}}\big)(\mu) \,=\, \D^{\LL}\eta \,-\, \ii_{\dd{h}}\mathscr{K}^{\LL},
    \end{equation}
then \ec{\mathrm{var}_{S}X_{H}} is a Hamiltonian derivation with respect to the Poisson bracket \ec{\{,\}^{\mathscr{\widetilde{L}}}} associated to the transversal given by \,\ec{\widetilde{\mathscr{L}} = (\mathrm{id}+\delta)(\LL)},\, where a vector bundle morphism \,\ec{\delta:\LL \rightarrow \T{S}}\, is defined in (\ref{RU3}). The corresponding Hamiltonian is given by \,\ec{H_{\widetilde{\mathscr{L}}}^{\operatorname{aff}} = \pi^{\ast}h + \ell_{(\eta^{\mathscr{L}}-\mu(v_{h}))}}.
\end{proposition}

Taking into account the relation
    \begin{equation*}
        \big(\D^{\mathscr{L}}\mu\big)(\alpha_{1},\alpha_{2}) \,=\, \D_{\alpha_{1}}^{\mathscr{L}}\mu(\alpha_{2}) \,-\, \D_{\alpha_{2}}^{\mathscr{L}}\mu(\alpha_{1}) \,-\, \mu\big([\alpha_{1},\alpha_{2}]_{\T^{\ast}S}\big),
    \end{equation*}
for \,\ec{\alpha_{1},\alpha_{2} \in \Gamma\,\T^{\ast}S},\, we represent equation (\ref{EcHamVar}) for $\mu$ in the intrinsic form
    \begin{equation}\label{1}
        \D_{\dd{h}}^{\mathscr{L}} \circ \mu \,-\, \mu \circ \mathrm{L}_{v_{h}} +\, \mathrm{ad}_{\eta} \circ \mu \,=\, \D^{\mathscr{L}}\eta^{\mathscr{L}} \,-\, \mathbf{i}_{\dd{h}}\mathscr{K}^{\mathscr{L}}.
    \end{equation}
Locally, this equation can be rewritten in terms of (local) vector fields \,\ec{\mu_{a} = \mu_{a}^{i}(x)\frac{\partial}{\partial x^{i}}}\, on $S$ as follows
    \begin{equation}\label{2}
        [v_{h},\mu_{b}] \,+\, \big(\mathbf{i}_{\dd{h}}\D_{b}^{a} - \lambda_{b}^{ac}\eta_{c}\big)\mu_{a} \,=\, -\, \Pi_{S}^{\natural}\dd_{S}\eta_{b} \,+\, \eta_{a}\D_{b}^{a} \,-\, \mathbf{i}_{\dd{h}}\mathscr{K}_{b},
    \end{equation}
where \,\ec{\D_{b}^{a}=\D_{b}^{ia}\frac{\partial}{\partial x^{i}}}\, and \,\ec{\mathscr{K}_{b} = \frac{1}{2}\mathscr{K}_{b}^{ij}\frac{\partial}{\partial x^{i}} \wedge \frac{\partial}{\partial x^{j}}}.\, If the normal bundle of $S$ is
trivial, then one can think of equations (\ref{2}) as a global matrix representation of (\ref{1}).

Finally, consider the case when a given contravariant derivative \,\ec{\D = \D^{\mathscr{L}}}\, admits representation (\ref{CC1}) for a
certain covariant derivative \,\ec{\nabla:\Gamma\,\T{S} \times \Gamma{E^{\ast}} \rightarrow \Gamma{E^{\ast}}}.\, Assume also that there exists a vector valued 2-form \,\ec{\mathscr{R}\in\Omega^{2}(S;E^{\ast})}\, such that the tensor field \,\ec{\mathscr{K}=\mathscr{K}^{\mathscr{L}}}\, is represented as
    \begin{equation*}
        \mathscr{K}(\alpha_{1},\alpha_{2}) \,=\, \mathscr{R}\big(\Pi_{S}^{\natural}\alpha_{1},\Pi_{S}^{\natural}\alpha_{2}\big),
    \end{equation*}
for \,\ec{\alpha_{1},\alpha_{2}\in\Gamma\,\T^{\ast}S}.\, Then, we have the following covariant version of equation (\ref{1}).

\begin{proposition}
If a vector valued 1-form \,\ec{\vartheta \in \Omega^{1}(S;E^{\ast})}\, satisfies the equation
    \begin{equation}\label{4}
        \nabla_{v_{h}}\vartheta \,-\, \vartheta \circ \mathrm{L}_{v_{h}} + [\eta^{\mathscr{L}},\vartheta]_{E^{\ast}} \,=\, \nabla\eta^{\mathscr{L}} - \mathbf{i}_{v_{h}}\mathscr{R},
    \end{equation}
then \,\ec{\mu = \vartheta \circ \Pi_{S}^{\natural}}\, is a solution to (\ref{1}).
\end{proposition}

Therefore, under above assumptions, the solvability of (\ref{4}) gives a sufficient condition for the Hamiltonization of the first variation system in the class of infinitesimal Poisson algebras.

In the case when $S$ is a symplectic leaf, the Poisson tensor \ec{\Pi_{S}} is nondegenerate and the solvability conditions for (\ref{1}) and (\ref{4}) are equivalent. The solvability of (\ref{4}) guaranties the existence of a Hamiltonian structure for \ec{\mathrm{var}_{S}X_{H}} in the class of coupling Poisson structures on $E$ \cite{Vor2001, Vor04}.

    \section{The Case of a Symplectic Leaf}\label{the case of a symplectic leaf}

Let \ec{(S,\omega_{S})} be an embedded symplectic leaf of \ec{(M,\Pi)}. So, the Poisson tensor \ec{\Pi_{S}} is nondegenerate and induces the symplectic form \ec{\omega_{S}} on $E$ defined by (\ref{SS}). As we mentioned above, in this case the Hamiltonization criterion for the first variation system \ec{\mathrm{var}_{S}X_{H}} can be formulated in a class of Poisson structures \cite{Vor05, Vor04}. First, we observe that contravariant derivative \ec{\D^{\LL}} induces a covariant derivative \,\ec{\nabla = \nabla^{\LL}}\, on \ec{E^{\ast}} given by (\ref{CC1}). Then, the adjoint derivative \ec{(\nabla^{\LL})^{\ast}} is a \emph{linear Poisson connection} on the normal bundle \ec{(E,\Lambda)}. Introducing the following antisymmetric mapping \,\ec{\sigma^{\LL}:\Gamma\,\T{M} \times \Gamma\,\T{M} \rightarrow \Cinf{\mathrm{aff}}(E)},
    \begin{equation}\label{EcSigmaL}
        \sigma^{\LL}(u_{1},u_{2}) \,:=\, \omega_{S}(u_{1},u_{2}) \,+\, \ell \circ \mathscr{K}^{\LL}\big((\Pi_{S}^{\natural})^{-1}u_{1},(\Pi_{S}^{\natural})^{-1}u_{2}\big),
    \end{equation}
we arrive at the following fact \cite{Vor2001}: in a neighborhood of the zero section \,\ec{S \hookrightarrow E},\, every transversal $\LL$ induces a Poisson tensor \ec{\Pi^{\LL}} defined as a \emph{coupling Poisson structure} associated with the geometric data \ec{((\nabla^{\LL})^{\ast},\sigma^{\LL},\Lambda)}.

Remark that in general, the coupling Lie bracket \ec{\{,\}^{\LL}} gives only a first-order approximation to the coupling Poisson structure \,\ec{\Pi^{\LL} = \Pi^{\LL}_{H} + \Lambda}\, in the sense that (see also \cite{Vor2001, Vor04})
    \begin{equation*}
        \Pi^{\LL}(\dd\phi_{1},\dd\phi_{2}) \,=\, \{\phi_{1},\phi_{2}\}^{\LL} + \OO_{2}.
    \end{equation*}
Here, \ec{\Pi^{\LL}_{H}} is the \ec{(\nabla^{\LL})^{\ast}}-horizontal part uniquely defined by \ec{\sigma^{\LL}}. One can show that the remainder in this equality vanishes if the zero curvature condition holds, \,\ec{\mathscr{K}^{\LL} \equiv 0}.\, In this case, the Lie bracket \ec{\{,\}^{\LL}} is canonically extended to a Poisson structure defined around the leaf $S$.

So, in the symplectic case, we have the following version of Theorem \ref{TeoHamPoissAlg}. \cite{Vor05}.

\begin{theorem}\label{TeoVarHam}
If a transversal $\LL$ is \ec{X_{H}}-invariant, then \ec{\mathrm{var}_{S}X_{H}} is a \emph{Hamiltonian vector field} on $E$ relative to the coupling Poisson structure \ec{\Pi^{\LL}} and the affine function \ec{H_{\LL}^{\mathrm{aff}}},
    \begin{equation}\label{EcVarXH}
        \mathrm{var}_{S}X_{H} \,=\, \ii_{\dd{H}_{\LL}^{\mathrm{aff}}}\Pi^{\LL}.
    \end{equation}
\end{theorem}
\begin{proof}
Consider the coupling Poisson tensor \ec{\Pi^{\LL}} associated to the data \,\ec{(\nabla^{\ast} = (\nabla^{\LL})^{\ast}, \sigma = \sigma^{\LL}, \Lambda)},
    \begin{equation*}
        \Pi^{\LL} = -\tfrac{1}{2}\sigma^{ij}\,\mathrm{hor}_{i}^{\nabla^{\ast}} \wedge \mathrm{hor}_{j}^{\nabla^{\ast}} + \Lambda, \qquad i,j = 1, \ldots, m.
    \end{equation*}
Here, \,\ec{\sigma^{is}\sigma_{sj} = \delta^{i}_{j}}, and \ec{\sigma_{ij}} are the components of the coupling form $\sigma$. Then, using the representation (\ref{HF1}) for \ec{\mathrm{var}_{S}X_{H}} and the relationship (\ref{EcSigmaL}) between \ec{\sigma^{\LL}} and \ec{\mathscr{K}^{\LL}}, by direct computation, we verify that condition (\ref{EcVarXH}) for \,\ec{H_{\LL}^{\mathrm{aff}} = \pi^{\ast}h + \ell_{\eta}}\, is just equivalent to the equation (\ref{HSFVS32}) for \ec{h \oplus \eta^{\LL}}. This fact together with Theorem \ref{TeoHamPoissAlg} and Lemma \ref{LemaIntrisicTorsion} ends the proof of the theorem.
\end{proof}

Finally, we formulate the following consequence of this result for the existence of linearized models of Hamiltonian group actions. Let \,\ec{\Phi:G \times M \rightarrow M}\, be a \emph{canonical action} of a connected Lie group $G$ on a Poisson manifold \ec{(M,\Pi)}, with a momentum map \,\ec{\mathrm{J}:M \rightarrow \mathfrak{g}^{\ast}},
    \begin{equation*}
        X_{a}\big|_{m} \,=\,\frac{\dd}{\dd t}\bigg|_{t=0}\big[\Phi_{\exp(ta)}(m)\big] \,=\, \Pi^{\natural}\dd\mathrm{J}_{a}\big|_{m}, \quad \forall\,a \in \mathfrak{g}.
    \end{equation*}
Then, the $G$-action leaves invariant a given (embedded) symplectic leaf \,\ec{S \subset M}\, and hence on the normal bundle \,\ec{\pi:E \rightarrow S},\, there exists an induced linearized $G$-action \,\ec{\varphi_{g}:E \rightarrow E}\, defined
by
    \begin{equation*}
        \big(\nu_{g \cdot m}\big)(\dd_{m}\Phi_{g}) \,=\, \varphi_{g} \cdot \nu_{m}, \quad m \in S,
    \end{equation*}
where \,\ec{\nu:\T_{S}M \rightarrow E}\, is the quotient projection.

\begin{theorem}\label{TeoAction}
If the $G$-action is proper, then there exists a $G$-invariant transversal \,\ec{\mathscr{L} \subset \T_{S}M}\, of $S$, and in a $G$-invariant neighborhood of $S$ in $E$, the \emph{linearized} $G$-action $\varphi$ is \emph{canonical} relative to the coupling Poisson structure \ec{\Pi^{\LL}} with fiberwise affine momentum map \,\ec{\mathrm{j}:E \rightarrow \mathfrak{g}^{\ast}}:
    \begin{equation*}
        \mathrm{var}_{S}X_{a} \,=\, \frac{\dd}{\dd t}\bigg|_{t=0}\big[\varphi_{\exp(ta)}\big] \,=\, \Pi^{\LL}\dd\,\mathrm{j}_{a},
    \end{equation*}
where \,\ec{\mathrm{j}_{a}=\mathrm{Aff}(\mathrm{J}_{a} \circ \mathbf{e}) \in \Cinf{\mathrm{aff}}(E)}.
\end{theorem}

The proof follows from Theorem \ref{TeoVarHam} and the fact \cite{DuKo00}: each proper action of a Lie group $G$ admits a $G$-invariant Riemannian metric on $M$. Then, a $G$-invariant transversal $\LL$ is defined as the orthogonal complement to $\T{S}$ in \ec{\T_{S}M}.

Notice that the assertion of Theorem \ref{TeoAction} is true when the Lie group $G$ is compact, since in this case, the action is proper.

    \subsection*{Acknowledgements}

This work was partially supported by the Mexican National Council of Science and Technology (CONACyT), under research project CB-258302. J. C. R.-P wishes also to thank CONACyT for the postdoctoral fellowship. The authors are very grateful to an anonymous Referee for critical comments and useful observations.

\end{document}